\documentclass[a4paper,11pt]{article}
\usepackage[all,ps,arc]{xy}  
\usepackage{amsmath,amssymb,amsthm,latexsym}
\usepackage{pdfsync}
\usepackage{makeidx}
\usepackage[english]{babel}
\usepackage[applemac]{inputenc}
\usepackage{graphicx}
\usepackage[x11names]{xcolor}
\usepackage{geometry}
\newtheorem{Theorem}{Theorem}[section]

\newtheorem{Proposition}[Theorem]{Proposition}
\newtheorem{Definition}[Theorem]{Definition}

\newtheorem{Remark}[Theorem]{Remark}
\newtheorem{Example}[Theorem]{Example}

\newtheorem{Condition}[Theorem]{Condition}
\newtheorem{Text}[Theorem]{}

\newcommand{\cA}{{\mathcal A}}
\newcommand{\cB}{{\mathcal B}}
\newcommand{\cC}{{\mathcal C}}
\newcommand{\cD}{{\mathcal D}}

\newcommand{\cF}{{\mathcal F}}
\newcommand{\cG}{{\mathcal G}}

\newcommand{\cK}{{\mathcal K}}
\newcommand{\cM}{{\mathcal M}}
\newcommand{\cN}{{\mathcal N}}

\newcommand{\cU}{{\mathcal U}}

\newcommand\Arr{\mathbf{Arr}}

\newcommand\Grpd{\mathbf{Grpd}}
\newcommand\RG{\mathbf{RG}}

\newcommand\Ker{\mathrm{Ker}}

\newcommand\id{\mathrm{id}}

\newcommand\Colim{\mathrm{Colim}}

\newcommand\HC{\mathrm{HoCok}}

\begin{document}

\title{Completion under strong homotopy cokernels}

\author{Enrico M. Vitale
\footnote{Institut de recherche en math\'ematique et physique, Universit\'e catholique de Louvain,
Chemin du Cyclotron 2, B 1348 Louvain-la-Neuve, Belgique, enrico.vitale@uclouvain.be}} 

\maketitle

\noindent {\bf Abstract:}
\noindent For $\cA$ a category with finite colimits, we show that the embedding of $\cA$ into the category of arrows $\Arr(\cA)$
determined by the initial object is the completion of $\cA$ under strong homotopy cokernels. The nullhomotopy structure
of $\Arr(\cA)$ (needed in order to express the notion of homotopy cokernel) is the usual one induced by the canonical string of 
adjunctions between $\cA$ and $\Arr(\cA).$  

\tableofcontents

\section{Introduction}\label{SecIntro}

Limits and colimits are a fundamental tool in category theory and its applications. However, these notions are not completely 
satisfactory in homotopical algebra, and the search for a convenient notion of homotopy limit is a long story, see for example
\cite{BK,Ma,TH,DF,DHKS}. 

More recently, (strong) homotopy kernels and (strong) homotopy cokernels in the context of categories equipped with a structure 
of nullhomotopies have been used in \cite{SnailEV,JMMV1,JMMV2,MGMV} in order to internalize Gabriel-Zisman \cite{GZ} and 
Brown \cite{BR} exact sequences, and in \cite{MMMV} to define a general notion of homotopy torsion theory.

The aim of the present paper is to exhibit the free completion of a category $\cA$ under strong homotopy cokernels. For this, we consider the
category $\Arr(\cA)$ of arrows in $\cA.$ The canonical embedding of $\cA$ in $\Arr(\cA)$ freely adds a factorization system to $\cA$
(see \cite{KT} and also \cite{GR00, RV}). If we assume that $\cA$ has an initial object $\emptyset,$ we can consider another embedding
given by the functor $\Gamma \colon \cA \to \Arr(\cA)$ which sends an object $X$ on the unique arrow $\emptyset \to X.$ We prove that, if $\cA$ 
has finite colimits and if we put on $\Arr(\cA)$ the structure of nullhomotopies induced by the canonical string of adjunctions between $\cA$ and 
$\Arr(\cA),$ then the functor $\Gamma$ is the free completion of $\cA$ under strong homotopy cokernels. If $\cA$ is finitely complete, by duality 
we get the free completion of $\cA$ under strong homotopy kernels.

The layout of the paper is as follows. In Section \ref{SecCatNull}, we recall the definition of category with nullhomotopies and we complete it 
with the appropriate notions of morphism and 2-morphism. We introduce also the examples relevant for this paper. More examples can be 
found in \cite{MMMV,DVfree}. Section \ref{SecHCok} is devoted to homotopy cokernels and to the behavior of colimits with respect to 
nullhomotopies. A particular attention is payed to the category $\Arr(\cA).$  Part of the material in Sections \ref{SecCatNull} and \ref{SecHCok}
is borrowed from the companion paper \cite{MMMV}. In Section \ref{SecUniv1}, we state in a precise way and prove the universality of the full 
embedding $\Gamma \colon \cA \to \Arr(\cA)$ mentioned above. In the last section, we discuss the denormalization functor from the point of 
view of the universal property of $\Arr(\cA).$
\hfill

\noindent N.B.: The composition of two arrows $\xymatrix{A \ar[r]^-{f} & B \ar[r]^-{g} & C}$ will be written as $f \cdot g.$

\section{Categories with nullhomotopies}\label{SecCatNull}

In this section, we fix the terminology and some basic facts concerning nullhomotopies. As far as I know, the notion of category with
a structure of nullhomotopies has been introduced in \cite{GR97}. I follow here the version, a bit stronger, adopted in \cite{SnailEV,JMMV1,MMMV}.

\begin{Definition}\label{DefNullHom}{\rm
A structure of nullhomotopies $\Theta$ on a category $\cB$ is given by the following data:
\begin{enumerate}
\item[1)] For every arrow $g$ in $\cB,$ a set $\Theta(g)$ whose elements are called nullhomotopies on $g.$
\item[2)] For every triple of composable arrows $\xymatrix{A \ar[r]^f & B \ar[r]^g & C \ar[r]^h & D},$ a map 
$$f \circ - \circ h \colon \Theta(g) \to \Theta(f \cdot g \cdot h)$$
in such a way that, for every $\varphi \in \Theta(g),$ one has
\begin{enumerate}
\item $(f' \cdot f) \circ \varphi \circ (h \cdot h') = f' \circ (f \circ \varphi \circ h) \circ h'$ whenever the compositions $f' \cdot f$ and $h \cdot h'$ 
are defined,
\item $\id_B \circ \varphi \circ \id_C = \varphi.$
\end{enumerate}
\end{enumerate}
When $f=\id_B$ or $h=\id_C,$ we write $\varphi \circ h$ and $f \circ \varphi$ instead of $\id_B \circ \varphi \circ h$ 
and $f \circ \varphi \circ \id_C.$
}\end{Definition}

\begin{Example}\label{ExNullHom}{\rm
In this paper, the relevant examples of structures of nullhomotopies are the first and the second example hereunder (and the dual of the 
first one). The third example is added in order to make clear in which sense a category with a structure of nullhomotopies can be seen 
as an intermediate notion between that of category and that of 2-category. Some examples having a 2-categorical flavor are discussed 
in \cite{DVfree}, where the quite involved passage from nullhomotopies to 2-cells in a 2-category is analyzed. Other examples are 
considered in \cite{MMMV}, where structures of nullhomotopies are obtained from generalized pre-(co)radicals, and where the link between 
structures of nullhomotopies and ideals of arrows is explained.
\begin{enumerate}
\item Let $\cA$ be a category with an initial object $\emptyset$ and write $\emptyset_C \colon \emptyset \to C$ for the unique arrow. 
We get a structure of nullhomotopies $\Theta_{\emptyset}$ on $\cA$ by taking as set of nullhomotopies on an arrow $g\colon B \to C$ the set
$$\Theta_{\emptyset}(g) = \{ \varphi \colon B \to \emptyset \mid \varphi \cdot \emptyset_C = g \}$$
Given arrows $f \colon A \to B$ and $h \colon C \to D,$ we put $f \circ \varphi \circ h = f \cdot \varphi$ for all $\varphi \in \Theta_{\emptyset}(g).$
\item Recall that, given a category $\cA,$ the category $\Arr(\cA)$ has as objects the arrows $b \colon B \to B_0$ of $\cA$ and as arrows pairs of 
arrows $(g,g_0)$ in $\cA$ such that
$$\xymatrix{B \ar[r]^{g} \ar[d]_{b} & C \ar[d]^{c} \\
B_0 \ar[r]_{g_0} & C_0}$$
commutes. As set of nullhomotopies $\Theta_{\Delta}(g,g_0)$ we take the set of diagonals:
$$\Theta_{\Delta}(g,g_0) = \{ \varphi \colon B_0 \to C \mid b \cdot \varphi = g, \; \varphi \cdot c = g_0 \}$$
In the situation of the following diagram
$$\xymatrix{A \ar[r]^{f} \ar[d]_{a} & B \ar[r]^{g} \ar[d]_{b} & C  \ar[r]^{h} \ar[d]^{c} & D \ar[d]^{d} \\
A_0 \ar[r]_{f_0} & B_0 \ar[r]_{g_0} \ar[ru]^{\varphi} & C_0 \ar[r]_{h_0} & D_0}$$
the composition is given by the formula 
$$(f,f_0) \circ \varphi \circ (h,h_0) = f_0 \cdot \varphi \cdot h$$
In \cite{MMMV}, it is shown that the structure $\Theta_{\Delta}$ on $\Arr(\cA)$ is the one induced by the string of adjunctions
$$\xymatrix{ \cA \ar[rr]|-{\cU} & & \Arr(\cA) \ar@<-1.5ex>[ll]_-{\cC} \ar@<1.5ex>[ll]^-{\cD} }
\;\;\;\;\; \cC \dashv \cU \dashv \cD$$
where $\cC$ is the codomain finctor, $\cD$ is the domain functor and $\cU$ is the full and faithful functor which sends an object 
$X$ on the identity arrow $\id_X.$
\item If the underlying category of a 2-category $\cB$ has zero object, then $\cB$ can be seen as a category with nullhomotopies 
by taking as nullhomotopies the 2-cells with domain a zero arrow (or the 2-cells with codomain a zero arrow). 
A relevant example which fits into this situation is discussed in Section \ref{SecDenorm}.
\end{enumerate}
}\end{Example}

\begin{Text}\label{TextNotNull}{\rm
The last item of Example \ref{ExNullHom} justifies the fact that, in a category with nullhomotopies $(\cB,\Theta),$ when a nullhomotopy 
$\varphi \in \Theta(g)$ is involved in a diagram, it will be depicted as
$$\xymatrix{B \ar@/^1.0pc/[rr]^{g} \ar@{-->}@/_1.0pc/[rr]_{0} & \Uparrow \varphi & C}$$
even if the category $\cB$ does not have zero arrows. For example, here there are the two ways to depict a nullhomotopy
$\varphi \in \Theta_{\Delta}(g,g_0)$ in $\Arr(\cA) \colon$
$$\xymatrix{B \ar[r]^{g} \ar[d]_{b} & C \ar[d]^{c} \\ 
B_0 \ar[ru]^{\varphi} \ar[r]_{g_0} & C_0} 
\;\;\;\;\; \mbox{ or } \;\;\;\;\;
\xymatrix{(B,b,B_0) \ar@/^1pc/[rr]^{(g,g_0)} \ar@{-->}@/_1pc/[rr]_{0} & \Uparrow \; \varphi & (C,c,C_0)}$$
}\end{Text}

\begin{Definition}\label{DefFunctNullHom}{\rm
(The 2-category of categories with nullhomotopies)
Let $(\cA,\Theta_{\cA})$ and $(\cB,\Theta_{\cB})$ be two categories with nullhomotopies.
\begin{enumerate}
\item[1)] A morphism $\cF \colon (\cA,\Theta_{\cA}) \to (\cB,\Theta_{\cB})$ is a functor $\cF \colon \cA \to \cB$ equipped, for every arrow 
$g \colon B \to C$ in  $\cA,$ with a map 
$$\cF_g \colon \Theta_{\cA}(g) \to \Theta_{\cal B}(\cF(g))$$
such that $\cF_{f \cdot g \cdot h}(f \circ \varphi \circ h) = \cF(f) \circ \cF_g(\varphi) \circ \cF(h)$ for all $f \colon A \to B$ and $h \colon C \to D.$ \\
\item[2)] If $\cG \colon (\cA,\Theta_{\cA}) \to (\cB,\Theta_{\cB})$ is another morphism, a 2-morphism $\alpha \colon \cF \Rightarrow \cG$ is a 
natural transformation such that, for every $g \colon B \to C$ in $\cA$ and for every $\varphi \in \Theta_{\cA}(g),$ one has 
$\alpha_B \circ \cG_g(\varphi) = \cF_g(\varphi) \circ \alpha_C.$
\end{enumerate}
(I will always omit the suffix $g$ in the map $\cF_{g}$ with the only exception of point 2) in the proof of Proposition \ref{PropExt}.)
}\end{Definition} 

\begin{Remark}\label{Rem2Cat}{\rm
Since morphisms compose as functors and since 2-morphisms compose vertically and horizontally as natural transformations, categories with 
nullhomotopies together with their morphisms and 2-morphisms form a 2-category. Observe also that, if a 2-morphism is invertible as a natural 
transformation, then the inverse natural transformation is automatically a 2-morphism. 
}\end{Remark}

\begin{Example}\label{ExFunctNullHom}{\rm 
If $\cA$ is a category with an initial object $\emptyset,$ we get a morphism of categories with nullhomotopies 
$\Gamma \colon (\cA,\Theta_{\emptyset}) \to (\Arr(\cA),\Theta_{\Delta})$ defined on objects, arrows and nullhomotopies by  
$$\xymatrix{ & \emptyset \ar[d] \\ B_0 \ar[ru]^{\varphi} \ar[r]_{g_0} & C_0} \;\; \mapsto \;\; 
\xymatrix{\emptyset \ar[r] \ar[d] & \emptyset \ar[d] \\ B_0 \ar[r]_{g_0} \ar[ru]^{\varphi} & C_0}$$
 The functor $\Gamma$ is full and faithful. Moreover, for every arrow $g_0 \colon B_0 \to C_0,$ the map 
 $\Gamma_{g_0} \colon \Theta_{\emptyset}(g_0) \to \Theta_{\Delta}(\Gamma(g_0))$ is bijective.
}\end{Example}

\begin{Condition}\label{CondRedInter}{\rm
Here we recall a condition crucial in this paper, but which is not always satisfied by a category with nullhomotopies. It has been isolated in \cite{GR01}
under the name of reduced interchange. We say that the reduced interchange holds in a category with nullhomotopies $(\cB,\Theta)$ if, 
in the situation
$$\xymatrix{A \ar@/^1.0pc/[rr]^{f} \ar@{-->}@/_1.0pc/[rr]_{0} & \Uparrow \alpha & B \ar@/^1.0pc/[rr]^{g} \ar@{-->}@/_1.0pc/[rr]_{0} & \Uparrow \beta & C}$$
one has that $\alpha \circ g = f \circ \beta.$
}\end{Condition}

\begin{Example}\label{ExRedInter}{\rm 
The reduced interchange holds in the examples of categories with nullhomotopies needed in this paper (see below). A more detailed analysis of this
condition can be found in \cite{DVfree}, where a simple counterexample is also given.
\begin{enumerate}
\item In $(\Arr(\cA),\Theta_{\Delta})$ the reduced interchange holds true. Indeed, given
$$\xymatrix{A \ar[d]_{a} \ar[r]^{f} & B \ar[d]_{b} \ar[r]^{g} & C \ar[d]^{c} \\
A_0 \ar[r]_{f_0} \ar[ru]^{\alpha} & B_0 \ar[r]_{g_0} \ar[ru]^{\beta} & C_0}$$
one has $\alpha \circ (g,g_0) = \alpha \cdot g = \alpha \cdot b \cdot \beta = f_0 \cdot \beta = (f,f_0) \circ \beta.$
\item Since the reduced interchange holds true in $(\Arr(\cA),\Theta_{\Delta}),$ the same happens in $(\cA,\Theta_{\emptyset}).$ 
This follows from the fact that the morphism $\Gamma$ of Example \ref{ExFunctNullHom} is bijective on nullhomotopies.
\item Let me notice here that, if the structure of nullhomotopies $\Theta$ in a category $\cB$ is the one induced by the unit of an idempotent monad 
or by the counit of an idempotent comonad on $\cB$ (see \cite{MMMV}), then the reduced interchange holds true in $(\cB,\Theta).$ The easy proof 
is left to the reader. The case of $(\Arr(\cA),\Theta_{\Delta})$ fits into this general remark because $\Theta_{\Delta}$ is induced by $\cC \dashv \cU$ 
or by $\cU \dashv \cD,$ as already recalled in Example \ref{ExNullHom}.
\end{enumerate}
}\end{Example} 

\section{Homotopy cokernels and strong colimits}\label{SecHCok}

A category with nullhomotopies does not have the 2-dimensional structure needed to express notions like 2-limits or bilimits. 
The convenient notions in the context of categories with nullhomotopies are those of (strong) homotopy kernels and (strong) 
homotopy cokernels. We copy the definition and the notation from \cite{MMMV}.

\begin{Definition}\label{DefHKer}{\rm
Let $g \colon B \to C$ be an arrow in a category with nullhomotopies $(\cB,\Theta).$ 
\begin{enumerate}
\item A homotopy cokernel of $g$ with respect to $\Theta$ (or $\Theta$-cokernel) is a triple 
$$\cC(g) \in \cB, c_g \colon C \to \cC(g), \gamma_g \in \Theta(g \cdot c_g)$$
such that, for any other triple $(D,h,\varphi \in \Theta(g \cdot h)),$ there exists a unique arrow $h'$ such that 
$c_g \cdot h' = h$ and $\gamma_g \circ h' = \varphi$
$$\xymatrix{ & \ar@{}[d]_{\gamma_g}|{\Downarrow} & \cC(g) \ar[dd]^{h'} \\ 
B \ar@/^1.0pc/@{-->}[rru]^{0} \ar@/_1.0pc/@{-->}[rrd]_{0} \ar[r]^-{g} & C \ar[ru]_{c_g} \ar[rd]^{h} \\
 & \ar@{}[u]_{\Uparrow}|{\varphi} & D }$$
\item A $\Theta$-cokernel $(\cC(g),c_g,\gamma_g)$ is strong if, for any triple $(D,h,\varphi \in \Theta(c_g \cdot h)),$ such that 
$g \circ \varphi = \gamma_g \circ h,$ there exists a unique nullhomotopy $\varphi' \in \Theta(h)$ such that $c_g \circ \varphi' = \varphi$
$$\xymatrix{B \ar[rr]^-{g} \ar@/_1.8pc/@{-->}[rrrr]_{0}^{\gamma_g\; \Uparrow} & & 
C \ar@/^1.8pc/@{-->}[rrrr]^{0}_{\varphi \; \Downarrow} \ar[rr]^-{c_g} & & \cC(g) \ar[rr]^{h} \ar@/_1.5pc/@{-->}[rr]_{0}^{\varphi' \; \Uparrow} & & D  }$$
\end{enumerate}
}\end{Definition}


\begin{Remark}\label{RemNofCof}{\rm
We list here some remarks on the $\Theta$-cokernel of an arrow in a category with nullhomotopies $(\cB,\Theta).$ 
\begin{enumerate}
\item Uniqueness: the $\Theta$-cokernel of an arrow is determined by its universal property uniquely up to a unique isomorphism. 
Moreover, if an arrow has two (necessarily isomorphic) $\Theta$-cokernels and one of them is strong, the other one also is strong.
\item Functoriality: in the situation of the following commutative solid diagram
$$\xymatrix{A \ar[rr]^{f} \ar[d]_{a} \ar@{-->}@/_2.5pc/[dd]_{0}^<<<<<<<<<<<<<<<<<<<{\Longrightarrow}
\ar@{-->}@/_2.5pc/[dd]_{0}^<<<<<<<<<<<<<<<<<{\; \gamma_a} & & 
B \ar[d]^{b}  \ar@{-->}@/^2.5pc/[dd]^{0}_<<<<<<<<<<<<<<<<<<<{\Longleftarrow}
\ar@{-->}@/^2.5pc/[dd]^{0}_<<<<<<<<<<<<<<<<<{\gamma_b \;} \\
A_0 \ar[rr]_{f_0} \ar[d]_{c_a} & & B_0 \ar[d]^{c_b} \\
\cC(a) \ar@{..>}[rr]_{\cC(f,f_0)} & & \cC(b)}$$
there exists a unique arrow $\cC(f,f_0) \colon \cC(a) \to \cC(b)$ such that $c_a \cdot \cC(f,f_0) = f_0 \cdot c_b$ and 
$\gamma_a \circ \cC(f,f_0) = f \circ \gamma_b.$
\item Behavior with respect to nullhomotopies: in the situation of the following commutative diagram 
$$\xymatrix{A \ar[rr]^{f} \ar[d]_{a} \ar@{-->}@/_2.5pc/[dd]_{0}^<<<<<<<<<<<<<<<<<<<{\Longrightarrow}
\ar@{-->}@/_2.5pc/[dd]_{0}^<<<<<<<<<<<<<<<<<{\; \gamma_a} & & 
B \ar[d]^{b}  \ar@{-->}@/^2.5pc/[dd]^{0}_<<<<<<<<<<<<<<<<<<<{\Longleftarrow}
\ar@{-->}@/^2.5pc/[dd]^{0}_<<<<<<<<<<<<<<<<<{\gamma_b \;} \\
A_0 \ar[rr]_{f_0} \ar[d]_{c_a} \ar[rru]^{d} & & B_0 \ar[d]^{c_b} \\
\cC(a) \ar[rr]^{\cC(f,f_0)}  \ar@{-->}@/_1.5pc/[rr]_{0}^{\Uparrow \; \cC(d) } & & \cC(b)}$$
if the $\Theta$-cokernel of the arrow $a$ is strong, then there exists a unique nullhomotopy $\cC(d) \in \Theta(\cC(f,f_0))$ 
such that $c_a \circ \cC(d) = d \circ \gamma_b.$
\item Cancellation properties:
\begin{enumerate}
\item In the situation
$$\xymatrix{A \ar@/^1.8pc/@{-->}[rrrr]^{0}_{\gamma_f \; \Downarrow}  \ar[rr]_{f} && B \ar[rr]_{c_f} && \cC(f) \ar@<0.5ex>[rr]^{g} \ar@<-0.5ex>[rr]_{h} && C}$$
if $c_f \cdot g = c_f \cdot h$ and $\gamma_f \circ g = \gamma_f \circ h,$ then $g=h.$
\item Assume now that the reduced interchange \ref{CondRedInter} holds in $(\cB,\Theta).$ In the situation
$$\xymatrix{A \ar[r]^{f} & B \ar[r]^{c_f} & \cC(f) \ar@/^1.8pc/@{-->}[rrr]^{0}_{\varphi \, \Downarrow \;\; \Downarrow \, \psi} \ar[rrr]_-{g}  &&& C}$$
if the $\Theta$-cokernel is strong and if the nullhomotopies $\varphi, \psi \in \Theta(g)$ are such that $c_f \circ \varphi = c_f \circ \psi,$ 
then $\varphi = \psi.$
\end{enumerate}
\end{enumerate}
}\end{Remark} 

\begin{proof}
We check point 4.(b) because this is the first place where we use the reduced interchange. Put $\alpha = c_f \circ \varphi.$
By the reduced interchange, we have $\gamma_f \circ g = f \cdot c_f \circ \varphi = f \circ \alpha.$ We can apply the universal property 
of the $\Theta$-cokernel and we get a unique nullhomotopy $\alpha' \in \Theta(g)$ such that $c_f \circ \alpha' = \alpha.$ Clearly, we can 
take $\alpha' = \varphi$ but, because of the hypothesis $c_f \circ \varphi = c_f \circ \psi,$ we can take also $\alpha' = \psi.$ By uniqueness 
of $\alpha',$ we are done.
\end{proof}

\begin{Remark}\label{RemCofCompl}{\rm 
Let us analyze here objects, arrows and nullhomotopies in $(\Arr(\cA),\Theta_{\Delta})$ from the point of view of $\Theta_{\Delta}$-cokernels.
In fact, the following simple remarks are the starting point to see that $(\Arr(\cA),\Theta_{\Delta})$ is the completion of $\cA$ by strong homotopy cokernels, as we will see in Section \ref{SecUniv1}.
\begin{enumerate}
\item Assume that the category $\cA$ has an initial object $\emptyset$ and consider the embedding $\Gamma \colon \cA \to \Arr(\cA)$ 
described in Example \ref{ExFunctNullHom}. For any arrow $a \colon A \to A_0$ in $\cA,$ the following diagram is a 
($\Theta_{\Delta}$-kernel $|$ $\Theta_{\Delta}$-cokernel) diagram in $(\Arr(\cA),\Theta_{\Delta}) \colon$
$$\xymatrix{\emptyset \ar[r] \ar[d] & \emptyset \ar[r] \ar[d] & A \ar[d]^{a} \\
A \ar[r]_{a} \ar[rru]^<<<<<<{\id} & A_0 \ar[r]_{\id} & A_0} \;\;\;\;\; \mbox{ that is } \;\;\;\;\; 
\xymatrix{ \\ \Gamma A \ar[rr]_{\Gamma a} \ar@/^1.8pc/@{-->}[rrrr]^{0}_{\id_A \; \Downarrow} & & 
\Gamma A_0 \ar[rr]_-{(\emptyset_A,\id_{A_0}} & & (A,a,A_0) }$$
In other words, each object $(a \colon A \to A_0)$ of $\Arr(\cA)$ is the $\Theta_{\Delta}$-cokernel of an arrow coming from $\cA$
(and each arrow of $\cA,$ once embedded in $\Arr(\cA),$ becomes the arrow part of a $\Theta_{\Delta}$-kernel). \\
\item More is true: each arrow $(f,f_0) \colon (A,a,A_0) \to (B,b,B_0)$ of $\Arr(\cA)$ is the unique extension to the $\Theta_{\Delta}$-cokernel 
(in the sense of Remark \ref{RemNofCof}.2) of a commutative square coming from $\cA,$ as in the following diagram:
$$\xymatrix{\Gamma A \ar[rr]^{\Gamma f} \ar[d]_{\Gamma a} \ar@{-->}@/_2.5pc/[dd]_{0}^<<<<<<<<<<<<<<<<<<<{\Longrightarrow}
\ar@{-->}@/_2.5pc/[dd]_{0}^<<<<<<<<<<<<<<<<<{\id_A} & & 
\Gamma B \ar[d]^{\Gamma b}  \ar@{-->}@/^2.5pc/[dd]^{0}_<<<<<<<<<<<<<<<<<<<{\Longleftarrow}
\ar@{-->}@/^2.5pc/[dd]^{0}_<<<<<<<<<<<<<<<<<{\id_B} \\
\Gamma A_0 \ar[rr]_{\Gamma f_0} \ar[d]^{(\emptyset_A,\id_{A_0})} & & \Gamma B_0 \ar[d]_{(\emptyset_B,\id_{B_0})} \\
(A,a,A_0) \ar[rr]_{(f,f_0)} & & (B,b,B_0)}$$
\item Finally, each nullhomotopy $\varphi \in \Theta_{\Delta}(f,f_0)$ is the unique extension to the $\Theta_{\Delta}$-cokernel 
(in the sense of Remark \ref{RemNofCof}.3) of a diagonal coming from $\cA,$ as in the following diagram:
$$\xymatrix{\Gamma A \ar[rr]^{\Gamma f} \ar[d]_{\Gamma a} \ar@{-->}@/_2.5pc/[dd]_{0}^<<<<<<<<<<<<<<<<<<<{\Longrightarrow}
\ar@{-->}@/_2.5pc/[dd]_{0}^<<<<<<<<<<<<<<<<<{\id_A} & & 
\Gamma B \ar[d]^{\Gamma b}  \ar@{-->}@/^2.5pc/[dd]^{0}_<<<<<<<<<<<<<<<<<<<{\Longleftarrow}
\ar@{-->}@/^2.5pc/[dd]^{0}_<<<<<<<<<<<<<<<<<{\id_B} \\
\Gamma A_0 \ar[rr]_{\Gamma f_0} \ar[d]^{(\emptyset_A,\id_{A_0})} \ar[rru]^{\Gamma \varphi} & & \Gamma B_0 \ar[d]_{(\emptyset_B,\id_{B_0})} \\
(A,a,A_0) \ar[rr]^{(f,f_0)}  \ar@{-->}@/_1.5pc/[rr]_{0}^{\Uparrow \; \varphi} & & (B,b,B_0)}$$
\end{enumerate}
}\end{Remark}

The following proposition appears in \cite{MMMV}, where it is deduced from some general results on the existence of homotopy cokernels.

\begin{Proposition}\label{PropCofNofArr} 
If a category $\cA$ has pushouts, then $\Arr(\cA)$ has strong $\Theta_{\Delta}$-cokernels.
\end{Proposition}

\begin{Text}\label{TextCofArr}{\rm
Even if Proposition \ref{PropCofNofArr} does not need a proof, I wish to share with the reader the guiding idea to construct  
$\Theta_{\Delta}$-cokernels in $\Arr(\cA)$ because it seems to me quite easy and instructive in order to understand the arguments 
behind the proof given in \cite{MMMV}. The following description already appears in \cite{SnailEV,GJ}. \\
The $\Theta_{\Delta}$-cokernel of an arrow $(f,f_0)$ in $\Arr(\cA)$ must be universal among all diagrams of shape
$$\xymatrix{A \ar[d]_{a} \ar[r]^{f} & B \ar[d]_<<<{b} \ar[r]^{g} & C \ar[d]^{c} \\
A_0 \ar[r]_{f_0} \ar[rru]_>>>>>>>>{\varphi} & B_0 \ar[r]_{g_0} & C_0}$$
where the following diagrams commute 
$$\xymatrix{A \ar[r]^{f} \ar[d]_{a} & B \ar[d]^{g} \\ A_0 \ar[r]_{\varphi} & C} \;\;\;\;\;\;\;\;\;\;\;\;\;
\xymatrix{A_0 \ar[d]_{\varphi} \ar[rrd]^<<<<<<{f_0} & & B \ar[lld]_<<<<<<{g} \ar[d]^{b} \\
C \ar[rd]_{c} & & B_0 \ar[ld]^{g_0} \\ & C_0}$$
So, just replace these two diagrams by the corresponding colimits. We get
$$\xymatrix{A \ar[rr]^{f} \ar[d]_{a} & & B \ar[d]^{a'} \\ A_0 \ar[rr]_{f'} & & A_0 +_{a,f} B} \;\;\;\;\;\;\;\;\;\;\;\;\;
\xymatrix{A_0 \ar[d]_{f'} \ar[rrd]^<<<<<<{f_0} & & B \ar[lld]_<<<<<<{a'} \ar[d]^{b} \\
A_0 +_{a,f} B \ar[rd]_{[f_0,b]} & & B_0 \ar[ld]^{\id} \\ & B_0}$$
where $A_0 +_{a,f} B$ is the pushout of $a$ and $f.$ Finally, the $\Theta_{\Delta}$-cokernel of $(f,f_0)$ is
$$\xymatrix{A \ar[rr]^{f} \ar[d]_{a} & & B \ar[d]_<<<{b} \ar[rr]^{a'} & & A_0 +_{a,f} B \ar[d]^{[f_0,b]} \\ 
A_0 \ar[rrrru]_>>>>>>>>>{f'} \ar[rr]_{f_0} & & B_0 \ar[rr]_{\id} & & B_0}$$
}\end{Text}


The interplay between nullhomotopies and colimits will enter in the statement and in the proof of the universal property of $\Arr(\cA).$ 
This is why we need the following definitions.

\begin{Definition}\label{DefNatNull}{\rm 
Consider two functors $\cF, \cG \colon \cD \to \cB,$ where $(\cB,\Theta)$ is a category with nullhomotopies. A natural nullhomotopy 
$$\xymatrix{\cD \ar@/^1.0pc/[rr]^{\cG} \ar@/_1.0pc/[rr]_{\cF} & \Uparrow \, \tau & \cB}$$
is given by a family of arrows and a family of nullhomotopies indexed by the objects of $\cD,$
$$\tau = \{\tau^a_D \colon \cF(D) \to \cG(D) \;,\;\; \tau^n_D \in \Theta(\tau^a_D) \}_{D \in \cD}$$
such that the family of arrows is a natural transformation and the family of nullhomotopies is such that 
$\tau^n_D \circ \cG(g) = \cF(g) \circ \tau^n_{D'}$ for all $g \colon D \to D'$ in $\cD.$
}\end{Definition}

\begin{Definition}\label{DefColimNull}{\rm
Consider a functor $\cF \colon \cD \to \cB,$ where $(\cB,\Theta)$ is a category with nullhomotopies, and write
$$\{ i_D \colon \cF(D) \to \mbox{colim}F \}_{D \in \cD}$$
for its colimit. We say that the colimit of $\cF$ is strong with respect to nullhomotopies (or $\Theta$-strong) if, 
for every object $X \in \cB$ and for every natural nullhmotopy
$$\xymatrix{\cD \ar@/^1.0pc/[rr]^{\kappa_X} \ar@/_1.0pc/[rr]_{\cF} & \Uparrow \, \tau & \cB}$$
($\kappa_X$ is the constant functor of value $X$) there exists a 
unique nullhomotopy $t^n \in \Theta(t^a)$ such that $i_D \circ t^n = \tau^n_D$ for all $D \in \cD,$ where 
$t^a \colon \mbox{colim}F \to X$ is the unique arrow such that $i_D \cdot t^a = \tau^a_D$ for all $D \in \cD.$ 
}\end{Definition}

\begin{Remark}\label{RemColimNull}{\rm
Let us make explicit two special cases of Definition \ref{DefColimNull}. The second one appears also in \cite{MMMV}.
Let $(\cB,\Theta)$ be a category with nullhomotopies.
\begin{enumerate}
\item An initial object $\emptyset$ is $\Theta$-strong if, for every object $X \in \cB,$ there is a unique nullhomotopy on the unique 
arrow $\emptyset_X \colon \emptyset \to X.$
\item Consider the factorization of a commutative square $f \cdot x = g \cdot y$ through the pushout of $f$ and $g$ as in the following diagram:
$$\xymatrix{A \ar[rr]^{g} \ar[d]_{f} & & C \ar[d]_{f'} \ar@/^2pc/[rrdd]^{y} \\
B \ar[rr]^{g'} \ar@/_1.5pc/[rrrrd]_{x} & & B+_{f,g}C \ar[rrd]^{[x,y]} \\
 & & & & D}$$
The pushout is $\Theta$-strong if, given two nullhomotopies $\varphi \in \Theta(x)$ and $\psi \in \Theta(y)$ such that $f \circ \varphi = g \circ \psi,$ 
there exists a unique nullhomotopy $[\varphi,\psi] \in \Theta([x,y])$ such that $g' \circ [\varphi,\psi) = \varphi$ and $f' \circ [\varphi,\psi] = \psi.$
\item Clearly, a $\Theta$-strong colimit has a cancellation property with respect to nullhomotopies. Here is the one for a $\Theta$-strong pushout 
(with the notation of the previous point): given an arrow $h \colon B+_{f,g}C \to D$ and nullhomotopies 
$\alpha, \beta \in \Theta(h),$ if $g' \circ \alpha = g' \circ \beta$ and $f' \circ \alpha = f' \circ \beta,$ then $\alpha = \beta.$
\end{enumerate}
}\end{Remark}

\begin{Example}\label{ExTrivCof}{\rm
Let $(\cB,\Theta)$ be a category with nullhomotopies and let $\emptyset$ be a $\Theta$-strong initial object in $\cB.$
If, for an object $X \in \cB,$ we call $\gamma_X \in \Theta(\emptyset_X)$ the unique nullhomotopy on $\emptyset_X,$ then the 
following diagram is a $\Theta$-cokernel:
$$\xymatrix{\emptyset \ar[rr]_{\emptyset_X} \ar@/^1.8pc/@{-->}[rrrr]^{0}_{\gamma_X \; \Downarrow} & & X \ar[rr]_-{\id_X} & & X }$$
 }\end{Example}

Here is the interplay between nullhomotopies and colimits in $\Arr(\cA).$

\begin{Proposition}\label{PropEnrLimArr}
Let $\cA$ be a category with finite colimlits.
\begin{enumerate}
\item Finite colimits in $\cA$ are $\Theta_{\emptyset}$-strong.
\item $\Arr(\cA)$ has finite colimits and they are $\Theta_{\Delta}$-strong.
\item The functor $\Gamma \colon \cA \to \Arr(\cA)$ preserves finite colimits.
\end{enumerate}
\end{Proposition}

\begin{proof}
The first point is an easy exercise. Moreover, colimits in $\Arr(\cA)$ are constructed level-wise from those in $\cA$ and 
obviously $\Gamma$ preserves colimits. Let me check, for example, that pushouts in $\Arr(\cA)$ are $\Theta_{\Delta}$-strong. 
Consider the following diagrams in $\Arr(\cA),$ the first one being a pushout and the second one being commutative :
$$\xymatrix{& A \ar[rr]^{g} \ar[dd]^<<<<<<{a} \ar[ld]_{f} & & C \ar[ld]_{f'} \ar[dd]^{c} \\
B \ar[rr]^<<<<<<<<<<<<<<<<<<<<{g'} \ar[dd]_{b} & & B+_{f,g}C \ar[dd]^<<<<<<{[b \cdot g_0', c \cdot f_0']} \\
& A_0 \ar[rr]^<<<<<<<<<<{g_0} \ar[ld]_{f_0} & & C_0 \ar[ld]^{f_0'} \\
B_0 \ar[rr]_{g_0'} & & B_0+_{f_0,g_0}C_0}
\;\;\;\;\;
\xymatrix{& A \ar[rr]^{g} \ar[dd]^<<<<<<{a} \ar[ld]_{f} & & C \ar[ld]_{y} \ar[dd]^{c} \\
B \ar[rr]^>>>>>>>>{x} \ar[dd]_{b} & & D \ar[dd]^<<<<<<{d} \\
& A_0 \ar[rr]^<<<<<<{g_0} \ar[ld]_{f_0} & & C_0 \ar[ld]^{y_0} \\
B_0 \ar[rr]_{x_0} & & D_0}$$
Consider also the unique factorization of the commutative diagram through the pushout:
$$\xymatrix{B+_{f,g}C \ar[d]_{[b \cdot g_0', c \cdot f_0']} \ar[rr]^{[x,y]} & & D \ar[d]^{d} \\
B_0+_{f_0,g_0}C_0 \ar[rr]_{[x_0,y_0]} & & D_0}$$
Given two nullhomotopies
$$\xymatrix{B \ar[r]^{x} \ar[d]_{b} & D \ar[d]_{d} & C \ar[l]_{y} \ar[d]^{c} \\
B_0 \ar[r]_{x_0} \ar[ru]^{\varphi} & D_0 & C_0 \ar[lu]_{\psi} \ar[l]^{y_0}}$$
the compatibility condition $(f,f_0) \circ \varphi = (g,g_0) \circ \psi$ means that $f_0 \cdot \varphi = g_0 \cdot \psi.$ Therefore, there exists a unique arrow
$[\varphi,\psi] \colon B_0+_{f_0,g_0}C_0 \to D$ such that $g_0' \cdot [\varphi,\psi] = \varphi$ and $f_0' \cdot [\varphi,\psi] = \psi.$
It remais to check that $[\varphi,\psi]$ is a nullhomotopy:
$$\xymatrix{B+_{f,g}C \ar[d]_{[b \cdot g_0', c \cdot f_0']} \ar[rr]^{[x,y]} & & D \ar[d]^{d} \\
B_0+_{f_0,g_0}C_0 \ar[rr]_{[x_0,y_0]} \ar[rru]^{[\varphi,\psi]} & & D_0}$$
The commutativity of the two triangles follows precomposing with the canonical arrows of the pushout.
\end{proof}

\section{Universality of $\Arr(\cA)$}\label{SecUniv1}

In this section we show that, if $\cA$ has finite colimits, the embedding $\Gamma \colon \cA \to \Arr(\cA)$ is the completion of $\cA$ 
by strong homotopy cokernels. We put on $\Arr(\cA)$ the structure of nullhomotopies $\Theta_{\Delta}$ introduced 
in Example \ref{ExNullHom}. The main point is to extend a functor $\cF \colon \cA \to \cB$ along $\Gamma \colon \cA \to \Arr(\cA).$ 

\begin{Proposition}\label{PropExt}
Consider a category $\cA$ with finite colimits, a category with nullhomotopies $(\cB,\Theta)$ satisfying the reduced interchange, 
and a functor $\cF \colon \cA \to \cB.$ Assume that
\begin{enumerate}
\item[(a)] the image by $\cF$ of finite colimits are $\Theta$-strong finite colimits, and
\item[(b)] the image by $\cF$ of any arrow in $\cA$ has a strong $\Theta$-cokernel in $\cB.$
\end{enumerate}
Under these conditions, there exists an essentially unique morphism of categories with nullhomotopies 
$\widehat \cF \colon (\Arr(\cA),\Theta_{\Delta}) \to (\cB,\Theta)$ sending
$\Theta_{\Delta}$-cokernels to strong $\Theta$-cokernels and such that $\Gamma \cdot \widehat \cF \simeq \cF.$ 
$$\xymatrix{\cA \ar[r]^<<<<<{\Gamma} \ar[rd]_{\cF} & \Arr(\cA) \ar[d]^{\widehat\cF} \\ & \cB}$$
Moreover, the image by $\widehat\cF$ of finite colimits are $\Theta$-strong finite colimits. 
\end{Proposition}

\begin{proof}
We split the proof into seven steps. \\
1) Construction of $\widehat \cF \colon$ start with two objects, an arrow and a nullhomotopy in $\Arr(\cA) \colon$
$$\xymatrix{A \ar[r]^{f} \ar[d]_{a} & B \ar[d]^{b} \\ 
A_0 \ar[ru]^{\alpha} \ar[r]_{f_0} & B_0}$$
their images by $\widehat \cF$ are depicted in the following commutative diagram, where both columns are $\Theta$-cokernels:
$$\xymatrix{\cF A \ar[rr]^{\cF f} \ar[d]_{\cF a} \ar@{-->}@/_3.0pc/[dd]_{0}^<<<<<<<<<<<<<{\Longrightarrow}
\ar@{-->}@/_3.0pc/[dd]_{0}^<<<<<<<<<<{\, \gamma_{\cF a}} & & 
\cF B \ar[d]^{\cF b}  \ar@{-->}@/^3.0pc/[dd]^{0}_<<<<<<<<<<<<<{\Longleftarrow}
\ar@{-->}@/^3.0pc/[dd]^{0}_<<<<<<<<<<{\gamma_{\cF b} \,} \\
\cF A_0 \ar[rr]_{\cF f_0} \ar[d]_{c_{\cF a}} \ar[rru]^{\cF \alpha} & & \cF B_0 \ar[d]^{c_{\cF b}} \\
\widehat \cF(A,a,A_0) \ar[rr]^{\widehat \cF(f,f_0)} \ar@{-->}@/_1.5pc/[rr]_{0}^{\Uparrow \; \widehat \cF \alpha} & & \widehat \cF(B,b,B_0)}$$
The arrow $\widehat \cF(f,f_0)$ is the unique arrow such that $c_{\cF a} \cdot \widehat \cF(f,f_0) = \cF f_0 \cdot c_{\cF b}$ and
$\gamma_{\cF a} \circ \widehat \cF(f,f_0) = \cF f \circ \gamma_{\cF b},$ see Remark \ref{RemNofCof}.2. 
The nullhomotopy $\widehat \cF \alpha$ is the unique nullhomotopy such that 
$c_{\cF a} \circ \widehat \cF \alpha = \cF \alpha \circ \gamma_{\cF b},$ see Remark \ref{RemNofCof}.3. 
It is easy to check that $\widehat\cF$ is indeed a morphism of categories with nullhomotopies. \\
2) Uniqueness of $\widehat \cF$ under the assumptions that $\widehat\cF$ preserves homotopy cokernels and extends $\cF$ 
along $\Gamma \colon$
consider once again a nullhomotopy in $\Arr(\cA)$
$$\xymatrix{A \ar[r]^{f} \ar[d]_{a} & B \ar[d]^{b} \\ 
A_0 \ar[ru]^{\varphi} \ar[r]_{f_0} & B_0}$$
Following Remark \ref{RemCofCompl}, we can present it as
$$\xymatrix{\Gamma A \ar[rr]^{\Gamma f} \ar[d]_{\Gamma a} \ar@{-->}@/_2.5pc/[dd]_{0}^<<<<<<<<<<<<<<<<<<<{\Longrightarrow}
\ar@{-->}@/_2.5pc/[dd]_{0}^<<<<<<<<<<<<<<<<<{\gamma_{\Gamma a}} & & 
\Gamma B \ar[d]^{\Gamma b}  \ar@{-->}@/^2.5pc/[dd]^{0}_<<<<<<<<<<<<<<<<<<<{\Longleftarrow}
\ar@{-->}@/^2.5pc/[dd]^{0}_<<<<<<<<<<<<<<<<<{\gamma_{\Gamma b}} \\
\Gamma A_0 \ar[rr]_{\Gamma f_0} \ar[d]_{c_{\Gamma a}} \ar[rru]^{\Gamma \varphi} & & \Gamma B_0 \ar[d]^{c_{\Gamma b}} \\
(A,a,A_0) \ar[rr]^{(f,f_0)}  \ar@{-->}@/_1.5pc/[rr]_{0}^{\Uparrow \; \varphi} & & (B,b,B_0)}$$
We have to compare what necessarily is  the image by $\widehat\cF$ of this diagram with the construction depicted in the 
first point of the proof. \\
(i) On objects: the first equality is due to the fact that $\widehat\cF$ preserves homotopy cokernels and the 
second one to the fact that $\widehat\cF$ extends $\cF$ along $\Gamma$
$$\resizebox{\displaywidth}{!}{
\xymatrix{ & \widehat\cF(\Gamma A_0) \ar[rd]^{\widehat\cF(c_{\Gamma a})} \\
\widehat\cF(\Gamma A) \ar[ru]^{\widehat\cF(\Gamma a)} \ar@{-->}[rr]_0 & \ar@{}[u]|{\Uparrow \; \widehat\cF(\gamma_{\Gamma a})} & \widehat\cF(A,a,A_0)}
=
\xymatrix{ & \widehat\cF(\Gamma A_0) \ar[rd]^{c_{\widehat\cF(\Gamma a)}} \\
\widehat\cF(\Gamma A) \ar[ru]^{\widehat\cF(\Gamma a)} \ar@{-->}[rr]_0 & \ar@{}[u]|{\Uparrow \; \gamma_{\widehat\cF(\Gamma a)}} & \cC(\widehat\cF(\Gamma a))}
=
\xymatrix{ & \cF A_0 \ar[rd]^{c_{\cF a}} \\
\cF A \ar[ru]^{\cF a} \ar@{-->}[rr]_0 & \ar@{}[u]|{\Uparrow \; \gamma_{\cF a}} & \cC(\cF a)}}$$
(ii) On arrows: we have to verify that our assumptions of $\widehat\cF$ force the equations
$$c_{\cF a} \cdot \widehat \cF(f,f_0) = \cF f_0 \cdot c_{\cF b} \;\;\; \mbox{ and } \;\;\; 
\gamma_{\cF a} \circ \widehat \cF(f,f_0) = \cF f \circ \gamma_{\cF b}$$ 
From Remark \ref{RemCofCompl}, we know that $c_{\Gamma a} \cdot (f,f_0) = \Gamma f_0 \cdot c_{\Gamma b}$ and
$\gamma_{\Gamma a} \circ (f,f_0) = \Gamma f \circ \gamma_{\Gamma b}.$
Therefore, by applying $\widehat\cF$ and using the conditions of Definition \ref{DefFunctNullHom}, we get
$$c_{Fa} \cdot \widehat\cF(f,f_0) = \widehat\cF(c_{\Gamma a}) \cdot \widehat\cF(f,f_0) = 
\widehat\cF(\Gamma f_0) \cdot \widehat\cF(c_{\Gamma b}) = Ff_0 \cdot c_{Fb}$$
$$\gamma_{Fa} \circ \widehat\cF(f,f_0) = \widehat\cF_{\Gamma a \cdot c_{\Gamma a}}(\gamma_{\Gamma a}) \circ \widehat\cF(f,f_0) =
\widehat\cF_{\Gamma a \cdot c_{\Gamma a} \cdot (f,f_0)}(\gamma_{\Gamma a} \circ (f,f_0)) =$$
$$= \widehat\cF_{\Gamma f \cdot \Gamma b \cdot c_{\Gamma b}}(\Gamma f \circ \gamma_{\Gamma b}) =
\widehat\cF(\Gamma f) \circ \widehat\cF_{\Gamma b \cdot c_{\Gamma b}}(\gamma_{\Gamma b}) = \cF f \circ \gamma_{\cF b}$$
(iii) On nullhomotopies: we have to verify that our assumptions of $\widehat\cF$ force the equation
$$c_{\cF a} \circ \widehat\cF_{(f,f_0)}(\varphi) = \cF \varphi \circ \gamma_{\cF b}$$
From Remark \ref{RemCofCompl}, we know that $c_{\Gamma a} \circ \varphi =  \Gamma \varphi \circ \gamma_{\Gamma b}.$
Therefore, by applying $\widehat\cF$ and using the conditions of Definition \ref{DefFunctNullHom}, we get
$$c_{\cF a} \circ \widehat\cF_{(f,f_0)}(\varphi) = \widehat\cF(c_{\Gamma a}) \circ \widehat\cF_{(f,f_0)}(\varphi) =
\widehat\cF_{c_{\Gamma a} \cdot (f,f_0)}(c_{\Gamma a} \circ \varphi) =$$
$$=\widehat\cF_{\Gamma \varphi \cdot \Gamma b \cdot c_{\Gamma b}}(\Gamma \varphi \circ \gamma_{\Gamma b}) =
\widehat\cF(\Gamma \varphi) \circ \widehat\cF_{\Gamma b \cdot c_{\Gamma b}}(\gamma_{\Gamma b}) = \cF \varphi \circ \gamma_{\cF b}$$
3) $\widehat\cF$ preserves homotopy cokernels: consider a $\Theta_{\Delta}$-cokernel in $\Arr(\cA)$ as in \ref{TextCofArr}
$$\xymatrix{A \ar[rr]^{f} \ar[d]_{a} & & B \ar[d]_<<<{b} \ar[rr]^{a'} & & A_0 +_{a,f} B \ar[d]^{[f_0,b]} \\ 
A_0 \ar[rrrru]_>>>>>>>>>{f'} \ar[rr]_{f_0} & & B_0 \ar[rr]_{\id} & & B_0}$$
and its image by $\widehat\cF$ (the three columns are $\Theta$-cokernels, but I omit from the picture the corresponding structural nullhomotopies
$\gamma_{\cF a}, \gamma_{\cF b}$ and $\gamma_{\cF[f_0,b]}$):
$$\xymatrix{\cF A \ar[d]_{\cF a} \ar[rr]^{\cF f} & & \cF B \ar[d]_<<<{\cF b} \ar[rr]^{\cF a'} & & \cF(A_0+_{a,f}B) \ar[d]^{\cF[f_0,b]} \\
\cF A_0 \ar[d]_{c_{\cF a}} \ar[rr]_{\cF f_0} \ar[rrrru]_>>>>>>>>>>>>>>{\cF f'} & & \cF B_0 \ar[d]_{c_{\cF b}} \ar[rr]_{\id} & & \cF B_0 \ar[d]^{c_{\cF[f_0,b]}} \\
\widehat\cF(A,a,A_0) \ar[rr]^{\widehat\cF(f,f_0)} \ar@{-->}@/_2.0pc/[rrrr]_{0}^{\Uparrow \;\; \widehat\cF f'} 
& & \widehat\cF(B,b,B_0) \ar[rr]^{\widehat\cF(a',\id)} & & \widehat\cF[f_0,b]}$$
We have to prove that the bottom row is a $\Theta$-cokernel. For this, consider a nullhomotopy in $\cB \colon$
 $$\xymatrix{ \widehat\cF(A,a,A_0) \ar[rr]_{\widehat\cF(f,f_0)} \ar@/^1.8pc/@{-->}[rrrr]^{0}_{\varphi \; \Downarrow} & & \widehat\cF(B,b,B_0) \ar[rr]_-{g} & & C }$$
We can construct two nullhomotopies in $\cB$
$$c_{\cF a} \circ \varphi \in \Theta(c_{\cF a} \cdot \widehat\cF(f,f_0) \cdot g) = \Theta(\cF f_0 \cdot c_{\cF b} \cdot g)
\;\;\mbox{ and }\;\;
\gamma_{\cF b} \circ g \in \Theta(\cF b \cdot c_{\cF b} \cdot g)$$
which satisfy the following condition (use Condition \ref{CondRedInter} for the first equality):
$$\cF a \cdot c_{\cF a} \circ \varphi = \gamma_{\cF a} \circ \widehat\cF(f,f_0) \cdot g = \cF f \circ \gamma_{\cF b} \circ g$$
Since, by assumption, the image by $\cF$ of a pushout is a $\Theta$-strong pushout, we can apply Remark \ref{RemColimNull}.2
and we get a unique nullhomotopy $\bar\varphi \in \Theta(\cF[f_0,b] \cdot c_{\cF b} \cdot g)$ such that $\cF f' \circ \bar\varphi = c_{\cF a} \circ \varphi$
and $\cF a' \circ \bar\varphi = \gamma_{\cF b} \circ g.$ Now, the existence of $\bar\varphi$ combined with the universal property of the
$\Theta$-cokernel $\widehat\cF[f_0,b]$ gives a unique arrow $g' \colon \widehat\cF[f_0,b] \to C$ such that $c_{\cF[f_0,b]} \cdot g' = c_{\cF b} \cdot g$
and $\gamma_{\cF[f_0,b]} \circ g' = \bar\varphi.$ We have to prove that the arrow $g'$ is the required factorization of $(g,\varphi)$ through 
$(\widehat\cF(a',\id),\widehat\cF f'),$ that is, $\widehat\cF(a',\id) \cdot g' = g$ and $\widehat\cF f' \circ g' = \varphi.$ We use, for both equations, 
Remark \ref{RemNofCof}.4. For the first one, precompose with $c_{\cF b}$ and $\gamma_{\cF b} \colon$
$$c_{\cF b} \cdot g = c_{\cF[f_0,b]} \cdot g' = c_{\cF b} \cdot \widehat\cF(a',\id) \cdot g'$$
$$\gamma_{\cF b} \circ g = \cF a' \circ \bar\varphi = \cF a' \circ \gamma_{\cF[f_0,b]} \circ g' = \gamma_{\cF b} \circ \widehat\cF(a',\id) \cdot g'$$
For the second one, precompose with $c_{\cF a} \colon$
$$c_{\cF a} \circ \varphi = \cF f' \circ \bar\varphi = \cF f' \circ \gamma_{\cF[f_0,b]} \circ g' = c_{\cF a} \circ \widehat\cF f' \circ g'$$
It remains to prove that the factorization $g'$ is unique. For this, assume that there is an arrow $\bar g \colon \widehat\cF[f_0,b] \to C$ such that
$\widehat\cF(a',\id) \cdot \bar g = g$ and $\widehat\cF f' \circ \bar g = \varphi.$ To prove that $\bar g = g'$ we have to prove that 
$c_{\cF[f_0,b]} \cdot \bar g = c_{\cF b} \cdot g$ and $\gamma_{\cF[f_0,b]} \circ \bar g = \bar\varphi.$ The verification of the first equation is direct:
$$c_{\cF b} \cdot g = c_{\cF b} \cdot \widehat\cF(a',\id) \cdot \bar g = c_{\cF[f_0,b]} \cdot \bar g$$
For the second equation, we go back to the conditions which define $\bar\varphi \colon$
$$\cF f' \circ \gamma_{\cF[f_0,b]} \circ \bar g = c_{\cF a} \circ \widehat\cF f' \circ \bar g = c_{\cF a} \circ \varphi$$
$$\cF a' \circ \gamma_{\cF[f_0,b]} \circ \bar g = \gamma_{\cF b} \circ \widehat\cF(a',\id) \cdot \bar g = \gamma_{\cF b} \circ g$$
4) The image by $\widehat\cF$ of a $\Theta_{\Delta}$-cokernel is a strong $\Theta$-cokernel: consider once again a $\Theta_{\Delta}$-cokernel 
in $\Arr(\cA)$ and its image by $\widehat\cF$ as at the beginning of point 3) of the proof. Consider also a nullhomotopy in $\cB$
$$\xymatrix{ \widehat\cF(B,b,B_0) \ar[rr]_{\widehat\cF(a',\id)} \ar@/^1.8pc/@{-->}[rrrr]^{0}_{\varphi \; \Downarrow} & & \widehat\cF[f_0,b] \ar[rr]_-{g} & & C }$$
and assume that $\varphi$ is compatible with $\widehat\cF f',$ that is, $\widehat\cF f' \circ g = \widehat\cF(f,f_0) \circ \varphi.$ 
We get a new nullhomotopy
$$c_{\cF b} \circ \varphi \in \Theta(c_{\cF b} \cdot \widehat\cF(a',\id) \cdot g) = \Theta(c_{\cF[f_0,b]} \cdot g)$$
In order to prove that $c_{\cF b} \circ \varphi$ is compatible with $\gamma_{\cF[f_0,b]},$ that is, 
$\cF[f_0,b] \cdot c_{\cF b} \circ \varphi = \gamma_{\cF[f_0,b]} \circ g,$
we use Remark \ref{RemColimNull}.3 once again, because $\cF$ sends pushouts to $\Theta$-strong pushouts:
$$\cF f' \cdot \cF[f_0,b] \cdot c_{\cF b} \circ \varphi = \cF f_0 \cdot c_{\cF b} \circ \varphi = c_{\cF a} \cdot \widehat\cF(f,f_0) \circ \varphi 
= c_{\cF a} \circ \widehat\cF f' \circ g = \cF f' \circ \gamma_{\cF[f_0,b]} \circ g$$
$$\cF a' \cdot \cF[f_0,b] \cdot c_{\cF b} \circ \varphi = \cF b \cdot c_{\cF b} \circ \varphi = \gamma_{\cF b} \circ \widehat\cF(a',\id) \cdot g
= \cF a' \circ \gamma_{\cF[f_0,b]} \circ g$$
Now, the universal property of the $\Theta$-cokernel $\widehat\cF[f_0,b]$ gives a unique nullhomotopy $\varphi' \in \Theta(g)$ such that 
$c_{\cF[f_0,b]} \circ \varphi' = c_{\cF b} \circ \varphi.$ We still have to check that $\varphi'$ is the required factorization, that is, 
$\widehat\cF(a',\id) \circ \varphi' = \varphi.$ Thanks to Remark \ref{RemNofCof}.4, it is enough to precompose with $c_{\cF b} \colon$
$$c_{\cF b} \cdot \widehat\cF(a',\id) \circ \varphi' = c_{\cF[f_0,b]} \circ \varphi' = c_{\cF b} \circ \varphi$$
It remains to prove that the factorization $\varphi'$ is unique. For this, assume that there is a nullhomotopy $\bar\varphi \in \Theta(g)$ 
such that $\widehat\cF(a',\id) \circ \bar\varphi = \varphi.$ To prove that $\bar\varphi = \varphi'$ we go back to the condition which defines $\varphi' \colon$
$$c_{\cF[f_0,b]} \circ \bar\varphi = c_{\cF b} \cdot \widehat\cF(a',\id) \circ \bar\varphi = c_{\cF b} \circ \varphi$$
5) $\widehat\cF$ extends $\cF$ along $\Gamma \colon$ by applying $\Gamma$ to an arrow $f \colon X \to Y$ in $\cA,$ we get
$$\xymatrix{\emptyset \ar[r]^{\id} \ar[d]_{\emptyset_X} & \emptyset \ar[d]^{\emptyset_Y} \\ X \ar[r]_{f} & Y}$$
and we have to compare the two diagrams hereunder, the first one giving the image of $\Gamma f$ by $\widehat\cF.$ If we can prove 
that the second one satisfies the conditions defining the first one, we can conclude that $\Gamma \cdot \widehat\cF = \cF.$
$$\xymatrix{ & \cF \emptyset \ar[rr]^{\id} \ar[d]^{\cF \emptyset_X} \ar@{-->}@/_3.2pc/[dd]_{0} & & \cF \emptyset \ar[d]_{\cF \emptyset_Y} \ar@{-->}@/^3.2pc/[dd]^{0} \\
\ar@{}[r]^{\gamma_{\cF \emptyset_X}}|{\Longrightarrow} & \cF X \ar[rr]^{\cF f} \ar[d]^{c_{\cF \emptyset_X}} & & 
\cF Y \ar[d]_{c_{\cF \emptyset_Y}}\ar@{}[r]^{\gamma_{\cF \emptyset_Y}}|{\Longleftarrow} & \\
 & \widehat\cF(\Gamma X) \ar[rr]_{\widehat\cF(\Gamma f)} & & \widehat\cF(\Gamma Y)}
 \;\;\;\;\;
 \xymatrix{ & \emptyset \ar[rr]^{\id} \ar[d]^{\emptyset_{\cF X}} \ar@{-->}@/_3.0pc/[dd]_{0} & & \emptyset \ar[d]_{\emptyset_{\cF Y}} \ar@{-->}@/^3.0pc/[dd]^{0} \\
\ar@{}[r]^{\gamma_{\cF X}}|{\Longrightarrow} & \cF X \ar[rr]^{\cF f} \ar[d]^{\id} & & 
\cF Y \ar[d]_{\id}\ar@{}[r]^{\gamma_{\cF Y}}|{\Longleftarrow} & \\
 & \cF X \ar[rr]_{\cF f} & & \cF Y}$$
Since, by assumption, $\cF$ sends the initial object of $\cA$ into a $\Theta$-strong initial object in $\cB,$ we can use 
Example \ref{ExTrivCof} and the columns of the second diagram are $\Theta$-cokernels. The equation $\id \cdot \cF f = \cF f \cdot \id$ is trivial.
Finally, the equation $\gamma_{\cF X} \circ \cF f = \id \circ \gamma_{\cF Y}$ follows once again from the fact that the initial 
object in $\cB$ is $\Theta$-strong. \\
6) $\widehat\cF$ preserves finite colimits: the preservation of the initial object follows from $\Gamma \cdot \widehat\cF \simeq \cF$ 
because both $\Gamma$ and $\cF$ preserve the initial. Consider now a pushout in $\Arr(\cA)$ (see the proof of Proposition 
\ref{PropEnrLimArr}) and its image by $\widehat\cF$ (I have omitted from the picture the structural nullhomotopies of the four columns, 
which are $\Theta$-cokernels): 
$$\xymatrix{& \cF A \ar[rr]^{\cF g} \ar[dd]^<<<<<<{\cF a} \ar[ld]_{\cF f} & & \cF C \ar[ld]_{\cF f'} \ar[dd]^{\cF c} \\
\cF B \ar[rr]^>>>>>>>>>>>>>>>>>>{\cF g'} \ar[dd]_{\cF b} & & \cF(B+_{f,g}C) \ar[dd]^<<<<<<{\cF(b+c)} \\
& \cF A_0 \ar[rr]^<<<<<<<<<<<<<<<<<<<<{\cF g_0} \ar[ld]_{\cF f_0} \ar[dd]_<<<<<<<{c_{\cF a}} & & \cF C_0 \ar[ld]^{\cF f_0'} \ar[dd]^{c_{\cF c}} \\
\cF B_0 \ar[rr]^>>>>>>>>>>>>>>>>{\cF g_0'} \ar[dd]_{c_{\cF b}} & & \cF(B_0+_{f_0,g_0}C_0) \ar[dd]^<<<<<<<{c_{\cF(b+c)}} \\
& \widehat\cF(A,a,A_0) \ar[rr]^<<<<<<<<<<<<<<<<{\widehat\cF(g,g_0)} \ar[ld]_{\widehat\cF(f,f_0)} 
& & \widehat\cF(C,c,C_0) \ar[ld]^<<<<<<<<{\widehat\cF(f',f_0')} \\
\widehat\cF(B,b,B_0) \ar[rr]_-{\widehat\cF(g',g_0')} & & \widehat\cF(B+_{f,g}C,b+c,B_0+_{f_0,g_0}C_0)}$$
We have to prove that the ground floor is a pushout in $\cB$ and we know, by assumption on $\cF,$ that the first and the second floor
are $\Theta$-strong pushouts. For this, consider two arrows
$$h \colon \widehat\cF(B,b,B_0) \to X \leftarrow \widehat\cF(C,c,C_0) \colon k$$
such that $\widehat\cF(f,f_0) \cdot h = \widehat\cF(g,g_0) \cdot k.$ Therefore
$$\cF f_0 \cdot c_{\cF b} \cdot h = c_{\cF a} \cdot \widehat\cF(f,f_0) \cdot h = 
c_{\cF a} \cdot \widehat\cF(g,g_0) \cdot k = \cF g_0 \cdot c_{\cF c} \cdot k$$
so that there exists a unique arrow $x \colon \cF(B_0+_{f_0,g_0}C_0) \to X$ such that $\cF g_0' \cdot x = c_{\cF b} \cdot h$ and 
$\cF f_0' \cdot x = c_{\cF c} \cdot k.$ We can now costrcut two nullhomotopies
$$\gamma_{\cF b} \circ h \in \Theta(\cF b \cdot c_{\cF b} \cdot h) = \Theta(\cF b \cdot \cF g_0' \cdot x) = \Theta(\cF g' \cdot \cF(b+c) \cdot x)$$
$$\gamma_{\cF c} \circ k \in \Theta(\cF c \cdot c_{\cF c} \cdot k) = \Theta(\cF c \cdot \cF f_0' \cdot x) = \Theta(\cF f' \cdot \cF(b+c) \cdot x)$$
which are compatible, indeed
$$\cF f \circ \gamma_{\cF b} \circ h = \gamma_{\cF a} \circ \widehat\cF(f,f_0) \cdot h = 
\gamma_{\cF a} \circ \widehat\cF(g,g_0) \cdot k = \cF g \circ \gamma_{\cF c} \circ k$$
Since the pushout $\cF(B+_{f,g}C)$ is $\Theta$-strong, we get a unique nullhomotopy $\psi \in \Theta(\cF(b+c) \cdot x)$ such that
$\cF g' \circ \psi = \gamma_{\cF b} \circ h$ and $\cF f' \circ \psi = \gamma_{\cF c} \circ k.$ By the universal property of the $\Theta$-cokernel, 
the nullhomotopy $\psi$ produces a unique arrow 
$$x' \colon \widehat\cF(B+_{f,g}C,b+c,B_0+_{f_0,g_0}C_0) \to X$$ 
such that $c_{\cF(b+c)} \cdot x' = x$ and $\gamma_{\cF(b+c)} \circ x' = \psi.$ We have to prove that $x'$ is the required factorization, that is,
$\widehat\cF(g',g_0') \cdot x' = h$ and $\widehat\cF(f',f_0') \cdot x' = k.$ We check the first condition (the second one is similar) using 
Remark \ref{RemNofCof}.4:
$$c_{\cF b} \cdot \widehat\cF(g',g_0') \cdot x' = \cF g_0' \cdot c_{\cF(b+c)} \cdot x' = \cF g_0' \cdot x = c_{\cF b} \cdot h$$
$$\gamma_{\cF b} \circ \widehat\cF(g',g_0') \cdot x' = \cF g' \circ \gamma_{\cF(b+c)} \circ x' = \cF g'  \circ \psi = \gamma_{\cF b} \circ h$$
It remains to prove that the factorization $x'$ is unique. For this, let 
$$\bar x \colon \widehat\cF(B+_{f,g}C,b+c,B_0+_{f_0,g_0}C_0) \to X$$ 
be an arrow such that $\widehat\cF(g',g_0') \cdot \bar x = h$ and $\widehat\cF(f',f_0') \cdot \bar x = k.$ In order to prove that $\bar x = x',$ 
we have to prove that $c_{\cF(b+c)} \cdot \bar x = x$ and $\gamma_{\cF(b+c)} \circ \bar x = \psi.$ For the first equation, we check the conditions 
which define $x \colon$
$$\cF g_0' \cdot c_{\cF(b+c)} \cdot \bar x = c_{\cF b} \cdot \widehat\cF(g',g_0') \cdot \bar x = c_{\cF b} \cdot h$$
$$\cF f_0' \cdot c_{\cF(b+c)} \cdot \bar x = c_{\cF c} \cdot \widehat\cF(f',f_0') \cdot \bar x = c_{\cF c} \cdot k$$
For the second equation, we check the conditions which define $\psi \colon$ 
$$\cF g' \circ \gamma_{\cF(b+c)} \circ \bar x = \gamma_{\cF b} \circ \widehat\cF(g',g_0') \cdot \bar x = \gamma_{\cF b} \circ h$$
$$\cF f' \circ \gamma_{\cF(b+c)} \circ \bar x = \gamma_{\cF c} \circ \widehat\cF(f',f_0') \cdot \bar x = \gamma_{\cF c} \circ k$$
7) The image by $\widehat\cF$ of finite colimits are $\Theta$-strong finite colimits: the case of the initial object is clear, so we pass to pushouts. 
We keep the same notations as in point 6). We have to prove that the pushout in $\cB$
$$\xymatrix{\widehat\cF(A,a,A_0) \ar[rrr]^{\widehat\cF(g,g_0)} \ar[d]_{\widehat\cF(f,f_0)} & & & \widehat\cF(C,c,C_0) \ar[d]^{\widehat\cF(f',f_0')} \\
\widehat\cF(B,b,B_0) \ar[rrr]_-{\widehat\cF(g',g_0')} & & & \widehat\cF(B+_{f,g}C,b+c,B_0+_{f_0,g_0}C_0)}$$
is $\Theta$-strong. For this, consider two nullhomotopies $\alpha \in \Theta(h)$ and $\beta \in \Theta(k)$ such that
$\widehat\cF(f,f_0) \circ \alpha = \widehat\cF(g,g_0) \circ \beta.$ It follows that
$$\cF f_0 \cdot c_{\cF b} \circ \alpha = c_{\cF a} \widehat\cF(f,f_0) \circ \alpha = 
c_{\cF a} \cdot \widehat\cF(g,g_0) \circ \beta = \cF g_0 \cdot c_{\cF c} \circ \beta$$
Since the pushout $\cF(B_0+_{f_0,g_0}C_0)$ is $\Theta$-strong, we get a unique nullhomotopy 
$[\alpha,\beta] \in \Theta(x)$ such that $\cF g_0' \circ [\alpha,\beta] = c_{\cF b} \circ \alpha$ and $\cF f_0' \circ [\alpha,\beta] = c_{\cF c} \circ \beta.$
Let us check that $\gamma_{\cF(b+c)} \circ x' = \cF(b+c) \circ [\alpha,\beta] \colon$ since the pushout $\cF(B+_{f,g}C)$ is $\Theta$-strong , 
we can use Remark \ref{RemColimNull}.3 and precompose with $\cF g'$ and $\cF f' \colon$
$$\cF g' \circ \gamma_{\cF(b+c)} \circ x' = \cF g' \circ \psi = \gamma_{\cF b} \circ h = 
\cF b \cdot c_{\cF b} \circ \alpha = \cF b \cdot \cF g_0' \circ [\alpha,\beta] = \cF g' \cdot \cF(b+c) \circ [\alpha,\beta]$$
$$\cF f' \circ \gamma_{\cF(b+c)} \circ x' = \cF f' \circ \psi = \gamma_{\cF c} \circ k = 
\cF c \cdot c_{\cF c} \circ \beta = \cF c \cdot \cF f_0' \circ [\alpha,\beta] = \cF f' \cdot \cF(b+c) \circ [\alpha,\beta]$$
By the universal property of the $\Theta$-cokernel  $\widehat\cF(B+_{f,g}C,b+c,B_0+_{f_0,g_0}C_0),$ we get a unique nullhomotopy
$\overline{[\alpha,\beta]} \in \Theta(x')$ such that $c_{\cF(b+c)} \circ \overline{[\alpha,\beta]} = [\alpha,\beta].$ We have to verify that
$\overline{[\alpha,\beta]}$ is the required extension, that is, $\widehat\cF(g',g_0') \circ \overline{[\alpha,\beta]} = \alpha$ and 
$\widehat\cF(f',f_0') \circ \overline{[\alpha,\beta]} = \beta.$ We check the first condition (the second one is similar) using
Remark \ref{RemNofCof}.4:
$$c_{\cF b} \cdot \widehat\cF(g',g_0') \circ \overline{[\alpha,\beta]} = \cF g_0' \cdot c_{\cF(b+c)} \circ \overline{[\alpha,\beta]} = 
\cF g_0' \circ [\alpha,\beta] = c_{\cF b} \circ \alpha$$
It remains to prove that the extension $\overline{[\alpha,\beta]}$ is unique. For this, let $\psi \in \Theta(x')$ be a nullhomotopy
such that $\widehat\cF(g',g_0') \circ \psi = \alpha$ and $\widehat\cF(f',f_0') \circ \psi = \beta.$ To show that $\psi = \overline{[\alpha,\beta]}$ 
it suffices to show that $c_{\cF(b+c)} \circ \psi = [\alpha,\beta].$ For this, we apply once again Remark \ref{RemColimNull}.3 to the
$\Theta$-strong pushout $\cF(B_0+_{f_0;g_0}C_0) \colon$
$$\cF g_0' \cdot c_{\cF(b+c)} \circ \psi = c_{\cF b} \cdot \widehat\cF(g',g_0') \circ \psi = c_{\cF b} \circ \alpha = \cF g_0' \circ [\alpha,\beta]$$
$$\cF f_0' \cdot c_{\cF(b+c)} \circ \psi = c_{\cF c} \cdot \widehat\cF(f',f_0') \circ \psi = c_{\cF c} \circ \beta = \cF f_0' \circ [\alpha,\beta]$$
The proof is now complete. 
\end{proof}

%

\begin{Text}\label{TextPropEquiv}{\rm
We restate now Proposition \ref{PropExt} in terms of an equivalence between hom-categories. 
Consider a category $\cA$ with finite colimits and a category with nullhomotopies $(\cB,\Theta)$ 
satisfying the reduced interchange. Assume that $\cB$ has $\Theta$-strong finite colimits 
and strong $\Theta$-cokernels. We are going to establish an equivalence between the following categories:
\begin{enumerate}
\item[-] $\Colim[\cA,\cB] \colon$ objects are functors preserving finite colimits, arrows are natural transformations,
\item[-] $\HC[\Arr(\cA),\cB] \colon$ objects are those morphisms $(\Arr(\cA),\Theta_{\Delta}) \to (\cB,\Theta)$
of Definition \ref{DefFunctNullHom}.1 which preserve finite colimits and 
homotopy cokernels, arrows are the 2-morphisms of Definition \ref{DefFunctNullHom}.2.
\end{enumerate}
}\end{Text}

\begin{Proposition}\label{PropEquiv}
Under the assumptions and with the notation of \ref{TextPropEquiv}, there is an equivalence of categories
$$\xymatrix{\HC[\Arr(\cA),\cB] \ar@<0.5ex>[rr]^-{\Gamma \cdot (-)} & & \Colim[\cA,\cB] \ar@<0.5ex>[ll]^-{\widehat{(-)}}}$$
\end{Proposition} 

\begin{proof}
We are going to prove that the functors $\Gamma \cdot (-)$ and $\widehat{(-)}$ are one the quasi-inverse of the other. \\
1) Definition of $\Gamma \cdot (-) \colon$ by Proposition \ref{PropEnrLimArr}, $\Gamma \colon \cA \to \Arr(\cA)$ preserves finite colimits, so that 
$\Gamma \cdot (-)$ is well-defined on objects. Its definition on arrows is obvious.\\
2) Definition of $\widehat{(-)} \colon$ from Proposition \ref{PropExt}, we already know how $\widehat{(-)}$ is defined on objects. As far as arrows are
concerned, consider a natural trasformation $\lambda \colon \cF \Rightarrow \cG$ in $ \Colim[\cA,\cB]$ and a nullhomotopy 
$$\xymatrix{A \ar[r]^{f} \ar[d]_{a} & B \ar[d]^{b} \\ 
A_0 \ar[ru]^{\varphi} \ar[r]_{f_0} & B_0}$$
in $\Arr(\cA).$ The next diagram describes the construction of $\widehat\lambda \colon \widehat\cF \Rightarrow \widehat\cG$ in $\HC[\Arr(\cA),\cB] \colon$
$$\xymatrix{ & \cF A \ar[rr]^{\lambda_A} \ar[d]^{\cF a} \ar@{-->}@/_3.3pc/[dd]_{0} & & \cG A \ar[d]_{\cG a} \ar@{-->}@/^3.3pc/[dd]^{0} \\
\ar@{}[r]^{\gamma_{\cF a}}|{\Longrightarrow} & \cF A_0 \ar[rr]^{\lambda_{A_0}} \ar[d]^{c_{\cF a}} & & 
\cG A_0 \ar[d]_{c_{\cG a}}\ar@{}[r]^{\gamma_{\cG a}}|{\Longleftarrow} & \\
 & \widehat\cF(A,a,A_0) \ar[rr]_{\widehat\lambda_{(A,a,A_0)}} & & \widehat\cG(A,a,A_0)}$$
In other words, $\widehat\lambda_{(A,a,A_0)}$ is the unique arrow such that 
$c_{\cF a} \cdot \widehat\lambda_{(A,a,A_0)} = \lambda_{A_0} \cdot c_{\cG a}$ and
$\gamma_{\cF a} \circ \widehat\lambda_{(A,a,A_0)} = \lambda_A \circ \gamma_{\cG a}.$ 
To check the naturality of $\widehat\lambda,$ precompose with $c_{\cF a}$ and $\gamma_{\cF a} \colon$
$$c_{\cF a} \cdot \widehat\cF(f,f_0) \cdot \widehat\lambda_{(B,b,B_0)} = \cF f_0 \cdot c_{\cF b} \cdot \widehat\lambda_{(B,b,B_0)} =
\cF f_0 \cdot \lambda_{B_0} \cdot c_{\cG b} =$$
$$= \lambda_{A_0} \cdot \cG f_0 \cdot c_{\cG b} =
\lambda_{A_0} \cdot c_{\cG a} \cdot \widehat\cG(f,f_0) = c_{\cF a} \cdot \widehat\lambda_{(A,a,A_0)} \cdot \widehat\cG(f,f_0)$$
$$\gamma_{\cF a} \circ \widehat\cF(f,f_0) \cdot \widehat\lambda_{(B,b,B_0)} = \cF f \circ \gamma_{\cF b} \circ \widehat\lambda_{(B,b,B_0)} =
\cF f \cdot \lambda_{B_0} \circ \gamma_{\cG b} =$$
$$= \lambda_{A} \cdot \cG f \circ \gamma_{\cG b} =
\lambda_{A} \circ \gamma_{\cG a} \circ \widehat\cG(f,f_0) = \gamma_{\cF a} \circ \widehat\lambda_{(A,a,A_0)} \cdot \widehat\cG(f,f_0)$$
To check that $\widehat\lambda$ is compatible with nullhomotopies in the sense of Definition \ref{DefFunctNullHom}.2, that is, 
$\widehat\lambda_{(A,a,A_0)} \circ \widehat\cG \varphi = \widehat\cF \varphi \circ \widehat\lambda_{(B,b,B_0)},$ precompose with $c_{\cF a} \colon$
$$c_{\cF a} \cdot \widehat\lambda_{(A,a,A_0)} \circ \widehat\cG \varphi = \lambda_{A_0} \cdot c_{\cG a} \circ \widehat\cG \varphi =
\lambda_{A_0} \cdot \cG \varphi \circ \gamma_{\cG b} =$$
$$= \cF \varphi \cdot \lambda_B \circ \gamma_{\cG b} = \cF \varphi \circ \gamma_{\cF b} \circ \widehat\lambda_{(B,b,B_0)} =
c_{\cF a} \circ \widehat\cF \varphi \circ \widehat\lambda_{(B,b,B_0)}$$
3) Composition $\xymatrix{\Colim[\cA,\cB] \ar[r]^-{\widehat{(-)}} & \HC[\Arr(\cA),\cB] \ar[r]^-{\Gamma \cdot (-)} & \Colim[\cA,\cB]} \colon$
from point 5) of the proof of Proposition \ref{PropExt}, we already 
know that $\Gamma \cdot \widehat\cF = \cF$ for any functor $\cF \in \Colim[\cA,\cB].$ Consider now a natural transformation 
$\lambda \colon \cF \Rightarrow \cG$ in $\Colim[\cA,\cB].$ We have to prove that the restriction along $\Gamma$ of $\widehat\lambda$ 
is $\lambda.$ This is because, if we start with an object $X \in \cA,$ the definition of  $\widehat\lambda_{\Gamma X}$ reduces to 
the following diagram (use Remark \ref{RemColimNull}.1):
$$\xymatrix{ & \emptyset \ar[rr]^{\id} \ar[d]^{\emptyset_{\cF X}} \ar@{-->}@/_3.0pc/[dd]_{0} & & \emptyset \ar[d]_{\emptyset_{\cG X}} \ar@{-->}@/^3.0pc/[dd]^{0} \\
\ar@{}[r]^{\gamma_{\cF X}}|{\Longrightarrow} & \cF X \ar[rr]^{\lambda_X} \ar[d]^{\id} & & 
\cG X \ar[d]_{\id}\ar@{}[r]^{\gamma_{\cG X}}|{\Longleftarrow} & \\
 & \cF X \ar[rr]_{\lambda_X} & & \cG X}$$
4) Composition $\xymatrix{\HC[\Arr(\cA),\cB] \ar[r]^-{\Gamma \cdot (-)} & \Colim[\cA,\cB] \ar[r]^-{\widehat{(-)}} & \HC[\Arr(\cA),\cB]} \colon$
we start with the construction, for any functor $ \cM \in \HC[\Arr(\cA),\cB],$ of an invertible 2-morphism 
$m \colon \widehat{\Gamma \cdot  \cM} \to  \cM.$ Its component at $(A,a,A_0) \in \Arr(\cA)$ is depicted in the following diagram:
$$\xymatrix{ & \cM \Gamma A \ar[d]_{ \cM\Gamma a} \ar@{-->}@/_2.0pc/[ldd]_{0} \ar@{-->}@/^2.0pc/[rdd]^{0} \\
\ar@{}[r]|{\Longrightarrow}^{\gamma_{ \cM \Gamma a}} &  \cM \Gamma A_0 \ar[ld]^{c_{ \cM\Gamma a}} \ar[rd]_{ \cM(0_A,\id_{A_0})} 
& \ar@{}[l]|{\Longleftarrow}_{ \cM(\id_A)} \\
\widehat{\Gamma \cdot  \cM}(A,a,A_0) \ar[rr]_{m_{(A,a,A_0)}} & &  \cM(A,a,A_0)}$$
The triangle on the left is a $\Theta$-cokernel by definition of $\widehat{\Gamma \cdot  \cM},$ the triangle on the right is a $\Theta$-cokernel by 
Remark \ref{RemCofCompl} and because $ \cM$ preserves $\Theta_{\Delta}$-cokernels. So, $m_{(A,a,A_0)}$ is the unique arrow such that 
$c_{ \cM\Gamma a} \cdot m_{(A,a,A_0)} =  \cM(0_A,\id_{A_0})$ and $\gamma_{ \cM\Gamma a} \cdot m_{(A,a,A_0)} =  \cM(\id_{A}).$ Moreover, 
$m_{(A,a,A_0)}$ is an isomorphism by Remark \ref{RemNofCof}.1. We have to prove that the family 
$$m = \{m_{(A,a,A_0)} \mid (A,a,A_0) \in \Arr(\cA)\}$$ 
is a 2-morphism in the sense of Definition \ref{DefFunctNullHom}. For this, consider a nullhomotopy
$$\xymatrix{A \ar[r]^{f} \ar[d]_{a} & B \ar[d]^{b} \\ 
A_0 \ar[ru]^{\varphi} \ar[r]_{f_0} & B_0}$$
in $\Arr(\cA).$ To check the naturality, precompose with $c_{ \cM\Gamma a}$ and $\gamma_{ \cM\Gamma a} \colon$
$$c_{ \cM\Gamma a} \cdot m_{(A,a,A_0)} \cdot  \cM(f,f_0) =  \cM(0_A,\id_{A_0}) \cdot  \cM(f,f_0) =  \cM\Gamma f_0 \cdot  \cM(0_B,\id_{B_0}) =$$ 
$$=  \cM\Gamma f_0 \cdot c_{ \cM\Gamma b} \cdot m_{(B,b,B_0)} = c_{ \cM\Gamma a} \cdot \widehat{\Gamma \cdot  \cM}(f,f_0) \cdot m_{(B,b,B_0)}$$
$$\gamma_{ \cM\Gamma a} \circ m_{(A,a,A_0)} \cdot  \cM(f,f_0) =  \cM(\id_A) \circ  \cM(f,f_0) =  \cM\Gamma f \circ  \cM(\id_B) =$$ 
$$=  \cM\Gamma f \circ \gamma_{ \cM\Gamma b} \circ m_{(B,b,B_0)} = \gamma_{ \cM\Gamma a} \circ \widehat{\Gamma \cdot  \cM}(f,f_0) \cdot m_{(B,b,B_0)}$$
To check the compatibility with nullhomotopies, precompose with $c_{ \cM\Gamma a} \colon$ 
$$c_{ \cM\Gamma a} \cdot m_{(A,a,A_0)} \circ  \cM(\lambda) =  \cM(0_A,\id_{A_0}) \circ  \cM(\lambda) =  \cM\Gamma \lambda \circ  \cM(\id_B) =$$
$$=  \cM\Gamma \lambda \circ \gamma_{ \cM\Gamma b} \circ m_{(B,b,B_0)} = c_{ \cM\Gamma a} \circ \widehat{\Gamma \cdot  \cM}(\lambda) \circ m_{(B,b,B_0)}$$
It remains to prove that, if $\mu \colon  \cM \Rightarrow  \cN$ is a 2-morphism in $\HC[\Arr(\cA),\cB],$ then
$$\xymatrix{\widehat{\Gamma \cdot  \cM} \ar[rr]^{m} \ar[d]_{\widehat{\Gamma \cdot \mu}} & &  \cM \ar[d]^{\mu} \\
\widehat{\Gamma \cdot  \cN} \ar[rr]_{n} & &  \cN}$$
commutes. This means that, for any object $(A,a,A_0) \in \Arr(\cA),$ we have to prove that 
$\widehat{\Gamma \cdot \mu}_{(A,a,A_0)} \cdot n_{(A,a,A_0)} = m_{(A,a,A_0)} \cdot \mu_{(A,a,A_0)}.$ By Remark \ref{RemNofCof}.4, it suffices to chek 
this equation by precomposing with $c_{ \cM\Gamma a}$ and $\gamma_{ \cM\Gamma a} \colon$
$$c_{ \cM\Gamma a} \cdot \widehat{\Gamma \cdot \mu}_{(A,a,A_0)} \cdot n_{(A,a,A_0)} = \mu_{\Gamma A_0} \cdot c_{ \cN\Gamma a} \cdot n_{(A,a,A_0)} =$$
$$= \mu_{\Gamma A_0} \cdot  \cN(0_A,\id_{A_0}) =  \cM(0_A,\id_{A_0}) \cdot \mu_{(A,a,A_0)} = c_{ \cM\Gamma a} \cdot m_{(A,a,A_0)} \cdot \mu_{(A,a,A_0)}$$
$$\gamma_{ \cM\Gamma a} \circ \widehat{\Gamma \cdot \mu}_{(A,a,A_0)} \cdot n_{(A,a,A_0)} = 
\mu_{\Gamma A} \circ \gamma_{ \cN\Gamma a} \circ n_{(A,a,A_0)} =$$
$$= \mu_{\Gamma A} \circ  \cN(\id_A) =  \cM(\id_A) \circ \mu_{(A,a,A_0)} = \gamma_{ \cM\Gamma a} \circ m_{(A,a,A_0)} \cdot \mu_{(A,a,A_0)}$$
The proof is now complete.
\end{proof}

\begin{Text}\label{TextIndHyp}{\rm
To end this section, let us point out that the assumptions on $(\cB,\Theta)$ appearing in \ref{TextPropEquiv} are not independent. Indeed, we know 
from \cite{MMMV} that, if $\cB$ has strong $\Theta$-cokernels of identity arrows and $\Theta$-strong pushouts, then it has all the $\Theta$-cokernels 
and  they are strong. Moreover, in the fundamental case where the structure $\Theta$ is induced by a string of adjunction 
 $$\xymatrix{ \cA \ar[rr]|-{\cU} & & \cB \ar@<-1.5ex>[ll]_-{} \ar@<1.5ex>[ll]^-{} }$$ 
with $\cU$ full and faithful and if $\cB$ has pushouts, then pushouts are $\Theta$-strong and $\cB$ has strong $\Theta$-cokernels.
}\end{Text} 

\section{The denormalization functor}\label{SecDenorm}

\begin{Text}\label{TextDual}{\rm
This short final section is completely devoted to illustrate, on a simple but relevant example, the extension 
$\widehat\cF$ of a functor $\cF \colon \cA \to \cB$ appearing in Proposition \ref{PropExt}, as well as the dual 
construction. As far as the dual constriuction 
is concerned, if we start assuming that $\cA$ has finite limits and we write $\ast$ for the terminal object and 
$\ast^B \colon B \to \ast$ for the unique arrow, the corresponding nullhomotopy structure on $\cA$ is
$\Theta_{\ast}(g) = \{ \varphi \colon \ast \to C \mid \ast^B \cdot \varphi = g \},$
the embedding $\Lambda \colon (\cA,\Theta_{\ast}) \to (\Arr(\cA),\Theta_{\Delta})$ is defined by  
$$\xymatrix{B \ar[r]^{g} \ar[d] & C \\ \ast \ar[ru]_{\varphi}} \;\; \mapsto \;\; 
\xymatrix{B \ar[r]^{g} \ar[d] & C \ar[d] \\ \ast \ar[r] \ar[ru]_{\varphi} & \ast}$$
and the extension along $\Lambda$ of a functor $\cF \colon \cA \to \cB$ is denoted by 
$\widetilde\cF \colon \Arr(\cA) \to \cB.$
}\end{Text} 

\begin{Text}\label{TextReflGr}{\rm
Starting from any category $\cA,$ we can construct the category $\RG(\cA)$ of reflexive graphs in $\cA.$ 
 Objects and arrows are depicted in the following diagram
 $$\xymatrix{A_1 \ar@<-0.5ex>[d]_{d} \ar@<0.5ex>[d]^{c}  \ar[rr]^{f_1} & & B_1  \ar@<-0.5ex>[d]_{d} \ar@<0.5ex>[d]^{c} \\
 A_0 \ar[rr]_{f_0} \ar@/^1.5pc/[u]^{i} & & B_0 \ar@/^1.5pc/[u]^{i}}$$
with the conditions $i \cdot d = \id = i \cdot c, \; i \cdot f_1 = f_0 \cdot i, \; d \cdot f_0 = f_1 \cdot d, \; c \cdot f_0 = f_1 \cdot c.$ \\
 If we assume that the category $\cA$ has a zero object and kernels, we can construct the so-called normalization functor 
 $\cK \colon \RG(\cA) \to \Arr(\cA)$ defined by
  $$\xymatrix{A_1 \ar@<-0.5ex>[d]_{d} \ar@<0.5ex>[d]^{c}  \ar[rr]^{f_1} & & B_1  \ar@<-0.5ex>[d]_{d} \ar@<0.5ex>[d]^{c} \\
 A_0 \ar[rr]_{f_0} \ar@/^1.5pc/[u]^{i} & & B_0 \ar@/^1.5pc/[u]^{i}}
 \;\;\;\;\;\; \mapsto \;\;\;\;\;\;
 \xymatrix{\Ker(d) \ar[r]^{K(f_1)} \ar[d]_{k_d} & \Ker(d) \ar[d]^{k_d} \\
 A_1 \ar@{.>}[r]^{f_1} \ar[d]_{c} & B_1 \ar[d]^{c} \\
 A_0 \ar[r]_{f_0} & B_0}$$ 
 where $K(f_1)$ is the unique arrow such that $K(f_1) \cdot k_d = k_d \cdot f_1.$
 A structure of nullhomotopies $\Theta$ on $\RG(\cA)$  can be chosen in such a way that $\cK$ is a morphism of categories 
 with nullhomotopies and it is bijective on nullhomotopies. Explicitly, a nullhomotopy on an arrow $(f_1,f_0)$ is an 
 arrow $\varphi \colon A_0 \to B_1$ such that $\varphi \cdot d = 0, \; \varphi \cdot c = f_0, \; k_d \cdot c \cdot \varphi = K(f_1) \cdot k_d.$ 
 }\end{Text} 
 
 \begin{Text}\label{TextPreDen}{\rm
 Now we construct two functors from $\cA$ to $\RG(\cA).$ The first one needs no assumption on $\cA.$
 For the second one, the existence of a zero object 0 is needed. Here they are:
 \begin{enumerate}
 \item[] $\Gamma' \colon \cA \to \RG(\cA) \;\;\;\;\;\;\;\;\;
\Gamma'(\xymatrix{ B_0 \ar[r]^{g_0} & C_0} ) = 
\xymatrix{& B_0 \ar@<-0.5ex>[d]_{\id} \ar@<0.5ex>[d]^{\id}  \ar[rr]^{g_0} & & C_0  \ar@<-0.5ex>[d]_{\id} \ar@<0.5ex>[d]^{\id} \\
& B_0 \ar[rr]_{g_0} \ar@/^1.5pc/[u]^{\id} & & C_0 \ar@/^1.5pc/[u]^{\id}}$
\item[] $\Lambda' \colon \cA \to \RG(\cA) \;\;\;\;\;\;\;\;\;
\Lambda'(\xymatrix{B \ar[r]^{g} & C}) = 
\xymatrix{& B \ar@<-0.5ex>[d]  \ar@<0.5ex>[d]  \ar[rr]^{g} & & C  \ar@<-0.5ex>[d] \ar@<0.5ex>[d] \\
& 0 \ar[rr] \ar@/^1.5pc/[u] & & 0 \ar@/^1.5pc/[u]}$
\end{enumerate} 
}\end{Text}

\begin{Text}\label{TextCof=Nof}{\rm
Assume now that $\cA$ is additive. 
The main point is to observe that the images by $\Gamma'$ and $\Lambda'$ of any arrow $a \colon A \to A_0$ of $\cA$ have, respectively, 
a $\Theta$-cokernel and a $\Theta$-kernel in $\RG(\cA).$ Moreover, the $\Theta$-cokernel of $\Gamma'(a)$ coincide with the $\Theta$-kernel 
of $\Lambda'(a).$ All this is depicted in the following diagram, where the dotted arrows are the structural nullhomotopies of the $\Theta$-cokernel 
(the one on the left) and of the $\Theta$-kernel (the one on the right):
$$\resizebox{\displaywidth}{!}{
\xymatrix{& A \ar@<-0.5ex>[d]_{\id} \ar@<0.5ex>[d]^{\id}  \ar[rr]^{a} & & A_0  \ar@<-0.5ex>[d]_{\id} \ar@<0.5ex>[d]^{\id} \ar[rr]^{i_1}
& & A_0 \oplus A \ar@<-0.5ex>[d]_{\pi_1} \ar@<0.5ex>[d]^{[\id;a]} \ar[rr]^{\pi_2} & & A \ar@<-0.5ex>[d] \ar@<0.5ex>[d] \ar[rr]^{a} 
& & A_0  \ar@<-0.5ex>[d] \ar@<0.5ex>[d] \\
& A \ar[rr]_{a} \ar@/^1.5pc/[u]^{\id} \ar@{.>}[rrrru]^<<<<<<<<<{i_2} & & A_0 \ar@/^1.5pc/[u]^{\id} \ar[rr]_{\id} 
& & A_0 \ar@/^1.5pc/[u]^{i_1} \ar[rr] \ar@{.>}[rrrru]^<<<<<<<<<<{\id} & & 0 \ar@/^1.5pc/[u] \ar[rr] & & 0 \ar@/^1.5pc/[u] }}$$
We can therefore extend $\Gamma'$ along $\Gamma$ and $\Lambda'$ along $\Lambda,$ as explained in the proof of Proposition 
\ref{PropExt}. In both cases, we get the so-called denormalization functor 
$$\cD \colon \Arr(\cA) \to \RG(\cA)$$ 
which sends an object $(A,a,A_0)$ on the reflexive graph in the middle of the previous diagram. It is well-known that $\cD$ is 
an equivalence of categories (with nullhomotopies) whose quasi-inverse is the normalization functor $\cK$ of \ref{TextReflGr}.
To prove this fact,
it is enough to check the following isomorphism between a reflexive graph and the denormalization of its normalization:
 $$\xymatrix{A_1 \ar@<-0.5ex>[d]_{d} \ar@<0.5ex>[d]^{c}  \ar@<0.5ex>[rrr]^-{\langle d;\delta \rangle} 
 & & & A_0 \oplus \Ker(d) \ar@<0.5ex>[lll]^-{[i;k_d]} \ar@<-0.5ex>[d]_{\pi_1} \ar@<0.5ex>[d]^{[\id;k_d \cdot c]} \\
 A_0 \ar@<0.5ex>[rrr]^{\id} \ar@/^1.5pc/[u]^{i} & & & A_0 \ar@<0.5ex>[lll]^{\id} \ar@/^1.5pc/[u]^{i_1} }$$
where $\delta \colon A_1 \to \Ker(d)$ is the unique arrow such that $\delta \cdot k_d = -d \cdot i + \id.$ 
}\end{Text}

\begin{Text}\label{Text2catStr}{\rm
Finally, we know that $\RG(\cA)$ is isomorphic to $\Grpd(\cA),$ the category of internal groupoids, because there 
exists a unique composition on the reflexive graphe $\cD(A,a,A_0)$ making it an internal category. It is given by
$$\id \oplus \nabla \colon A_0 \oplus A \oplus A \to A_0 \oplus A$$
(see \cite{AC, PTJ} for a detailed discussion). Transporting the 2-categorical structure of $\Grpd(\cA)$ along 
$\Grpd(\cA) \simeq \RG(\cA) \simeq \Arr(\cA),$ we get a 2-categorical structure on $\Arr(\cA)$ which extends
the structure of nullhomotopies $\Theta_{\Delta} \colon$ for any arrow $(f,f_0) \colon (A,a,A_0) \to (B,b,B_0),$
the set of nullhomotopies $\Theta_{\Delta}((f,f_0))$ coincides with the set of 2-cells from the zero arrow 
$(0^A_B, 0^{A_0}_{B_0})$ to $(f,f_0).$
 }\end{Text}

\end{document}